\documentclass{amsart}

\usepackage{amsfonts}
\usepackage{amssymb}
\usepackage{amsmath}
\usepackage{graphicx}
\usepackage{amsthm}
\usepackage{xcolor}
\usepackage{hyperref}
\usepackage{nicematrix}

\usepackage{float}

\newcommand{\bF}{\mathbb{F}}
\newcommand{\bN}{\mathbb{N}}
\newcommand{\bP}{\mathbb{P}}

\newcommand{\bfA}{\mathbf{A}}
\newcommand{\bfB}{\mathbf{B}}
\newcommand{\bfM}{\mathbf{M}}

\newcommand{\bfX}{\mathbf{X}}
\newcommand{\bfY}{\mathbf{Y}}

\newcommand{\bfu}{\mathbf{u}}

\newcommand{\bfx}{\mathbf{x}}
\newcommand{\bfy}{\mathbf{y}}
\newcommand{\bfz}{\mathbf{z}}

\newcommand{\cL}{\mathcal{L}}
\newcommand{\cR}{\mathcal{R}}
\newcommand{\cS}{\mathcal{S}}
\newcommand{\cT}{\mathcal{T}}

\newcommand{\Ff}{\widehat{f}}

\newcommand{\mn}{{\rm M}_n}
\newcommand{\mnq}{{\rm M}_n(\mathbb{F}_q)}

\def\({\left(}
\def\){\right)}

\numberwithin{equation}{section}

\newtheorem{thm}{Theorem}[section]
\newtheorem*{thm*}{Remark}

\newtheorem{prop}[thm]{Proposition}
\newtheorem{lem}[thm]{Lemma}
\newtheorem{rem}[thm]{Remark}

\newtheorem{conj}[thm]{Conjecture}

\makeatletter
\@namedef{subjclassname@2010}{%
  \textup{2010} Mathematics Subject Classification}
\makeatother

\begin{document}

\title{A matrix variant of the Erd\H{o}s-Falconer distance problems over finite field}

\author{Hieu T. Ngo}
\address{Institute of Mathematics, 
Vietnam Academy of Science and Technology, Hanoi, Vietnam}
\email{nthieu@math.ac.vn}

\thanks{}

\begin{abstract} 
We study a matrix analog of the Erd\H{o}s-Falconer distance problems in vector spaces over finite fields. There arises an interesting analysis of certain quadratic matrix Gauss sums. 
\end{abstract}

\maketitle

\section{Introduction}

\subsection{Problem formulation}\label{subsect:formulation}

Let $\bF_q$ be a finite field of odd cardinality $q$. 
On the vector space $\bF_q^r$, consider the norm-like function 
$$
\|\bfx\|= \sum_{i=1}^r x_i^2 \in \bF_q 
$$
where $\bfx=(x_1,\dots,x_r)\in\bF_q^r$. On $\bF_q^r \times \bF_q^r$, define the distance-like function 
$$
d(\bfx,\bfy)=\|\bfx-\bfy\| 
\quad (\bfx,\bfy \in \bF_q^r).
$$
The finite field Erd\H{o}s distance problem seeks, for a subset $E\subseteq\bF_q^r$, the smallest possible cardinality of the `distance set' 
$$
\Delta(E)=\{d(\bfx,\bfy): \bfx,\bfy \in E\}.
$$

Throughout this paper, the cardinality $q$ of $\bF_q$ is an asymptotic parameter that can be arbitrarily large.
For two functions $f$ and $g$ of odd prime powers $q$ which take values in the complex numbers, we adopt the Vinogradov notation $f\gg g$ (resp.~$f\ll g$) to mean that there exists a positive constant $c>0$ such that $|f(q)|>c|g(q)|$ (resp.~$|f(q)|< c|g(q)|$) for all such $q$; the Bachmann-Landau notations $f=O(g)$ and $g=\Omega(f)$ have the same meaning as $f\ll g$. If both asymptotics $f\ll g$ and $g\ll f$ are satisfied, one  writes $f\asymp g$ or, equivalently, $f=\Theta(g)$. In addition, by $f\sim g$ we mean $\frac{f(q)}{g(q)}$ tends to $1$ as $q$ approaches infinity.

The following finite field Falconer distance conjecture was proposed by Iosevich and Rudnev \cite[Conjecture 1.1]{IR07}.

\begin{conj}[Iosevich-Rudnev]\label{conj:IR}
    Suppose that $r\geq 2$ and that $E\subseteq \bF_q^r$ has cardinality at least $cq^{\frac{r}{2}}$ with $c$ sufficiently large.
    Then $\#\Delta(E) \gg q$.
\end{conj}

The lower bound exponent $\frac{r+1}{2}$ was the first result toward this conjecture \cite[Theorem 1.3]{IR07}.

\begin{thm}[Iosevich-Rudnev]\label{thm:IR}
    Suppose that $r\geq 2$ and that $E\subseteq \bF_q^r$ has cardinality at least $cq^{\frac{r+1}{2}}$ where $c>0$ is sufficiently large.
    Then $\Delta(E) = \bF_q$.
\end{thm}

An alternative proof, which bypassed Kloosterman sums, of Theorem \ref{thm:IR} was presented in \cite{AMM17}. 
The Erd\H{o}s-Falconer distance problems in vector spaces over finite fields exhibit the important philosophy of finite field models in arithmetic combinatorics \cite{Green05,Wolf15}. 
The work \cite{IR07} has stimulated many extensions and generalizations, albeit still in the context of vector spaces over finite fields (see, for instance, \cite{IK08,Vinh11EF,KS12,Dietmann13}).

There is a further direction that is worth inquiring: the generalization from commutativity to noncommutativity.
To our best knowledge, no noncommutative version of the Erd\H{o}s-Falconer distance problems has been formulated in the literature.
In this paper, we propose a matrix analog of the finite field Erd\H{o}s-Falconer distance problems as follows. Let $n$ and $r$ be positive integers. Let $\mn=\mnq$ denote the (noncommutative) algebra of $n\times n$ matrices over $\bF_q$.
On the free $\mn$-module $\mn^r$, consider the norm-like function 
$$
\|\bfX\|= \sum_{i=1}^r X_i^2 \in \mn 
$$
where $\bfX=(X_1,\dots,X_r)\in\mn^r$. On $\mn^r \times \mn^r$, define the distance-like function 
$$
d(\bfX,\bfY)=\|\bfX-\bfY\| 
\quad (\bfX,\bfY \in \mn^r).
$$
The matrix Erd\H{o}s distance problem seeks, for a subset $E\subseteq\mn^r$, the smallest possible cardinality of the `distance set' 
$$
\Delta(E)=\{d(\bfX,\bfY): \bfX,\bfY \in E\}.
$$
The matrix Falconer distance problem is to determine the smallest exponent $e=e(n,r)$ such that, if $r\geq 2$ and $E\subseteq \mn^r$ has cardinality at least $cq^{e}$ with $c$ sufficiently large, then $\#\Delta(E) \gg q^{n^2}$.



\subsection{New results}\label{subsect:new-results}

Our goal is to give a nontrivial lower bound for the cardinality of a subset such that its distance set is the whole matrix algebra. Our first two results achieve this goal when the matrix algebra has small rank $2$ or $3$.

\begin{thm}\label{thm:strong-Falconer-rank2}
    Let $\bF_q$ be a finite field of odd cardinality $q$.
    Let $E$ be a subset of $\left({\rm M}_2(\bF_q)\right)^r$ with $r\geq 4$.
    For every $\epsilon>0$, there exists a constant $c=c(\epsilon)>0$ such that
    $$\Delta(E)={\rm M}_2(\bF_q)$$
    whenever $(\# E) > cq^{3r+3+\epsilon}$.
\end{thm}

\begin{thm}\label{thm:strong-Falconer-rank3}
    Let $\bF_q$ be a finite field of odd cardinality $q$.
    Let $E$ be a subset of $\left({\rm M}_3(\bF_q)\right)^r$ with $r\geq 3$.
    For every $\epsilon>0$, there exists a constant $c=c(\epsilon)>0$ such that
    $$\Delta(E)={\rm M}_3(\bF_q)$$
    whenever $(\# E) > cq^{7r+4+\epsilon}$.
\end{thm}

Our next result concerns a matrix algebra of general high rank.

\begin{thm}\label{thm:strong-Falconer-general-rank}
    Let $\bF_q$ be a finite field of odd cardinality $q$.
    Let $E$ be a subset of $\left({\rm M}_n(\bF_q)\right)^r$ with $n\geq 3$ and $r\geq 3$.
    For every $\epsilon>0$, there exists a constant $c=c(\epsilon)>0$ such that
    $$\Delta(E)={\rm M}_n(\bF_q)$$
    whenever $(\# E) > cq^{rn^2-(r-2)(n-1)+\epsilon}$.
\end{thm}

\begin{rem}
    Theorem \ref{thm:strong-Falconer-rank3} can be derived as a corollary of Theorem \ref{thm:strong-Falconer-general-rank}. However, we will prove Theorem \ref{thm:strong-Falconer-rank3} before Theorem \ref{thm:strong-Falconer-general-rank}. We believe that the proof of the former helps one familiarize with various ingredients, such as invariants of quadratic cycle types, that go into the proof of the latter. 
\end{rem}

\subsection{}\label{subsect:organization}

The paper is organized as follows.
In Section \ref{sect:prelim}, we collect relevant notions and results concerning discrete Fourier analysis and similarity classes of matrices over a finite field.
Section \ref{sect:Gauss-sums} studies a matrix analog of the classical quadratic Gauss sums.  
We then analyze `matrix spheres' over a finite field, thereby proving Theorems \ref{thm:strong-Falconer-rank2} and \ref{thm:strong-Falconer-rank3} in Section \ref{sect:proofs-ranks-23} and Theorem \ref{thm:strong-Falconer-general-rank} in Section \ref{sect:proofs-general-rank}.


\section{Preliminaries}\label{sect:prelim}

\subsection{Discrete Fourier analysis on matrices}\label{subsect:2-Fourier}

In this section we recall basic notions of discrete Fourier analysis on matrix spaces over a finite field.

Let $\bF_q$ be a finite field with $q=p^g$ elements. 
Let $\mn:=\mnq$ be the algebra of $n\times n$ matrices over $\bF_q$.
Denote by ${\rm Tr}_{\bF_q/\bF_p}$ the standard trace map.
The set $\mn$ with matrix addition is a finite abelian group and has a nontrivial additive character
$$
\psi(X)=\exp\(\frac{2\pi i}{p} \, {\rm Tr}_{\bF_q/\bF_p}\({\rm Tr} \, X\)\) 
    \quad (X\in \mn)
$$
where ${\rm Tr}\,X$ is the matrix trace of $X$. 
Let $\widehat{\rm M}_n$ denote the Pontryagin dual of $\mn$, i.e.~the group of all additive characters of $\mn$. The map 
$$
\mn \to \widehat{\rm M}_n, A\mapsto\psi_A
$$
where $\psi_A(X)=\psi(AX)\, (X\in\mn)$, is a group isomorphism and gives a parametrization of $\widehat{\rm M}_n$. One has the orthogonal relation
$$
\frac{1}{q^{n^2}}\sum_{A\in\mn} \psi(AX) = \delta_0(X)
    \quad (X\in\mn)
$$
where $\delta_0(X)$ equals $1$ if $X=0$ and equals $0$ otherwise.

Let $r$ be a positive integer. The free $\mn$-module $\mn^r$ is equipped with the dot product
$$ \bfA \cdot \bfB = \sum_{i=1}^r A_iB_i
$$
where $\bfA=(A_1,\dots,A_r) \in \mn^r$ and $\bfB=(B_1,\dots,B_r)\in \mn^r$.
The Fourier transform of a complex-valued function $f$ on $\mn^r$ is given by
$$
\Ff(\bfM) = \frac{1}{q^{rn^2}} \sum_{\bfA\in\mn^r} 
    f(\bfA) \overline{\psi}(\bfM\cdot \bfA) 
    \quad (\bfM\in\mn^r).
$$
The Fourier inversion formula expands $f$ in terms of $\Ff$:
$$
f(\bfA)= \sum_{\bfM\in\mn^r}  \Ff(\bfM) \psi(\bfM\cdot \bfA).
$$
The Plancherel formula relates the $L^2$-norms of $f$ and $\Ff$:
$$
\frac{1}{q^{rn^2}} \sum_{\bfA\in\mn^r} |f(\bfA)|^2
    = \sum_{\bfM\in\mn^r} |\Ff(\bfM)|^2.
$$

\subsection{Similarity of matrices}
\label{subsect:2-sim}

In this section, let $k$ be an arbitrary field and consider the conjugate action of ${\rm GL}_n(k)$ on ${\rm M}_n(k)$. Two matrices in the same orbit are said to be \emph{similar}; each orbit is called a \emph{similarity class}.

\subsubsection*{Rational canonical forms}
\label{sssect:2-RCF}

The theory of rational canonical forms asserts that a square matrix over $k$ is classified up to conjugation by a sequence of monic polynomials 
$$f_1(T),\dots,f_s(T) \in k[T]$$ 
satisfying $f_{i+1}(T)|f_i(T)$ for all $1\leq i<s$.
More precisely, for a monic polynomial $f(T)\in k[T]$ define the companion matrix $C_f$ as follows. If $f(T)=T+c$, set $C_f=(-c)$. If $f(T)=T^r+c_{r-1}T^{r-1}+\cdots+c_1T+c_0$, let $C_f$ be the $r\times r$ matrix 
$$
C_f=\begin{pmatrix}  
    0 & 0 & \cdots & 0 &   -c_0 \\
    1 & 0 & \cdots & 0 &  -c_1 \\
    0 & 1 & \cdots & 0 &  -c_2 \\
    \cdots \\
    0 & 0 & \cdots & 1 &  -c_{r-1} \\
\end{pmatrix};
$$
this is the matrix corresponding to the `multiplication by $T$' linear transformation on the vector space $\frac{k[T]}{(f(T))}$. 
Any matrix $A\in {\rm M}_n(k)$ is conjugate (by a matrix in ${\rm GL}_n(k)$) to a block diagonal matrix 
$$ R_A = 
\begin{pmatrix}  
    C_{f_1} & 0 & \cdots & 0 \\
    0 & C_{f_2} & \cdots & 0 \\
    \cdots \\
    0 & 0 & \cdots & C_{f_{s}} 
\end{pmatrix}
$$
where $f_{i+1}(T)|f_i(T)$ for all $1\leq i<s$.
The monic polynomials $\{f_i(T): 1\leq i\leq s\}$ are the \emph{invariant factors} of $A$; the matrix $R_A$ is the \emph{rational canonical form} of $A$. Two matrices in ${\rm M}_n(k)$ are similar if and only if they have the same invariant factors. We refer the reader to \cite[Section 9.2]{Rotman10} for a thorough and clear exposition of the theory of rational canonical forms.

\subsubsection*{Cycle types}
\label{sssect:2-cycle-type}

Similarity classes of matrices over a field can also be described in terms of cycle types as follows. Let ${\rm Irr}(k[T])$ denote the set of monic irreducible polynomials in $k[T]$. For any matrix $A\in \mn(k)$, one defines a $k[T]$-module structure on $k^n$ by $T\cdot v = Av$ for $v\in k^n$; this is called the \emph{operator module} associated to $A$ and denoted by $M_A$. The operator modules $M_A$ and $M_B$ are isomorphic if and only if the matrices $A$ and $B$ are similar. For $\pi \in {\rm Irr}(k[T])$, the $\pi$-primary part of $M_A$ is 
$$
M_{A,\pi}=\{v\in M_A: \pi^rv=0 \,\, \text{for some} \,\, r\in\bN\}.
$$
There is a decomposition
$$
M_{A,\pi} \cong \bigoplus_i \frac{k[T]}{(\pi^{\lambda_{\pi,i}})}
$$
where the $\lambda_{\pi,i}$ are positive integers.
Let $\Lambda$ denote the set of all partitions. 
The \emph{cycle type} of $A$ is the map 
$$
\nu_A: {\rm Irr}(k[T]) \to \Lambda
$$
given by $\nu_A(\pi)=\lambda_{\pi}=(\lambda_{\pi,i})_i$. 
Put $|\lambda_{\pi}|:=\sum_i \lambda_{\pi,i}$.
One can express the cycle type of $A$ as a formal product
$$
\nu_A = \prod_{\pi} \pi^{\lambda_{\pi}}.
$$
Define the \emph{degree of the formal product} $\prod_{\pi} \pi^{\lambda_{\pi}}$ to be $\sum_{\pi} \deg(\pi)|\lambda_{\pi}|$; note that the degree of the cycle type $\nu_A$ is nothing but $n$.
The cycle type $A\mapsto \nu_A$ gives a bijection between similarity classes in $\mn(k)$ and the set of all maps from ${\rm Irr}(k[T])$ to $\Lambda$ which have degree $n$.

\subsubsection*{Class types and centralizers}
\label{sssect:2-class-type}

Over a finite field $k=\bF_q$, Green \cite{Green55} introduced the notion of (similarity) class type. If $A\in\mnq$ has cycle type
$\nu_A = \pi_1^{\lambda_1} \cdots \pi_s^{\lambda_s}$ and for each $i$ the polynomial $\pi_i$ has degree $d_i$, the \emph{(similarity) class type} of $A$ is the formal product
$$
\tau_A = d_1^{\lambda_1} \cdots d_s^{\lambda_s}.
$$
Green \cite[Lemma 2.1]{Green55} discovered that the class type of a matrix determines its centralizer up to isomorphism. Furthermore, Britnell and Wildon \cite[Theorem 2.7]{BW11} showed that two matrices have the same class type if and only if their centralizers are conjugate (by an element of ${\rm GL}_n(\bF_q)$). Both of these results made essential use of the finite field assumption; see \cite{BW14} for a generalization to an arbitrary field. 

From the class type of a matrix, one can compute the size of its centralizer, thanks to the classical works of Kung \cite{Kung81}, Stong \cite{Stong88}, and Fulman \cite{Fulman99}. Let $A\in\mnq$ have class type
$$
\tau_A = d_1^{\lambda_1} \cdots d_s^{\lambda_s}.
$$
Dropping the subscript, we write $d^{\lambda}$ for an arbitrary component $d_i^{\lambda_i}$ of $\tau_A$. 

For a partition $\lambda$ of a positive integer into parts $\lambda_{1}\geq\lambda_2\geq \cdots$, 
set
\begin{align*}
    |\lambda| &= \sum_i \lambda_i, \\
    |\lambda|_2 &= \sum_i \lambda_i^2. 
\end{align*}
Let $m_j(\lambda)$ be the number of parts in the partition $\lambda$ which are equal to $j$. 
The conjugate partition $\lambda'$ of $\lambda$ has its $j^{\rm th}$ part given by
$$
\lambda'_j =  m_j(\lambda) + m_{j+1}(\lambda) + \cdots .
$$
Write
$$
(1/x)_r=\prod_{j=1}^r\(1-1/x^j\).
$$
Define (cf.~\cite[Section 2]{Fulman99}, \cite[Theorem 9.5]{PSS15})
\begin{align}
c\(d^\lambda\) 
    &= q^{d |\lambda'|_2} \prod_j (1/q^d)_{m_j(\lambda)}.  \label{eq:cent-size-1} 
    \\
    &= q^{d \sum_{j}(\lambda'_j)^2}
    \prod_j (1/q^d)_{m_j(\lambda)}.   \nonumber
\end{align}
Then the centralizer in ${\rm GL}_n(\bF_q)$ of $A$ has size 
\begin{equation}\label{eq:cent-size-2}
c(\tau_A)= c\(d_1^{\lambda_1} \cdots d_s^{\lambda_s}\)
    = \prod_{i=1}^s c\(d_i^{\lambda_i}\).
\end{equation}


\section{Quadratic matrix Gauss sums}\label{sect:Gauss-sums}

Let $p$ be an odd prime and $\bF_q$ be a finite field with $q=p^g$ elements. 
Write $\mn=\mnq$ for the algebra of $n\times n$ matrices over $\bF_q$.
For $A,B\in\mn$, define the quadratic matrix Gauss sum 
\begin{equation}\label{eq:quadratic-matrix-Gauss-sum}
G(A,B) = \sum_{X\in\mn} \psi(AX^2+BX).    
\end{equation}

\subsection{The trace quadratic form}\label{subsect:31}

For $A\in\mn$, the trace form ${\rm Tr}(AX^2)$ defines a quadratic space $(V_A,Q_A)$ of dimension $n^2$ over $\bF_q$. In a series of papers \cite{Kuroda97,Kuroda99,Kuroda04}, Kuroda studied this `trace quadratic form' and applied it to compute $G(A,B)$ in the special case $B=0$. 

Suppose that $A$ is conjugate by an element of ${\rm GL}_n\(\overline{\bF}_q\)$ to a Jordan canonical form 
$$ 
\begin{pmatrix}  
    A_{1} & 0 & \cdots & 0 \\
    0 & A_{2} & \cdots & 0 \\
    \cdots \\
    0 & 0 & \cdots & A_{r} 
\end{pmatrix}
$$
where $A_i$ is the $n_i \times n_i$ Jordan block with eigenvalue $\alpha_i$. The radical of $V_A$, denoted ${\rm rad}\,A$, has dimension \cite[Lemma 3]{Kuroda04}
\begin{equation}\label{eq:radical-invariant}
\rho_A = \sum_{\substack{1\leq i,j \leq r \\ \alpha_i+\alpha_j=0}}
    \min\(n_i,n_j\).
\end{equation}
Let us call $\rho_A$ the \emph{radical invariant} of the matrix $A$.

\begin{prop}\label{prop:Kuroda}\cite[Lemma 4]{Kuroda04} 
Suppose that the characteristic polynomial of $A$ has a factorization into irreducible polynomials
$$
P_A(T) = T^{c} \cdot \prod_{i=1}^{r} \pi_i\(T^2\)^{a_i}
    \cdot \prod_{j=1}^{s}\zeta_j(T)^{b_j}
$$
where for each $1\leq j\leq s$ the polynomial $\zeta_j(T)$ is not a polynomial in $T^2$.
Then one has an orthogonal decomposition
$$
V_A \cong W_{\zeta_1}^{b_1} \perp \cdots  \perp W_{\zeta_s}^{b_s} \perp H \perp {\rm rad} \,A
$$
where $W_{\zeta_j}$ is a regular subspace of $V_A$ depending on $\zeta_j(T)$ with dimension 
$$ \dim \, W_{\zeta_j} = \deg\, \zeta_j(T) \quad (1\leq j \leq s),
$$
and $H$ is a hyperbolic space.
\end{prop}


\subsection{Estimating quadratic matrix Gauss sums}\label{subsect:32}

Recall the following beautiful exponential sum estimate of Deligne \cite[Théorèm 8.4]{Deligne74} (see also \cite{Katz99}).

\begin{thm}[Deligne]\label{thm:Deligne}
Let $f(x_1,\dots,x_n)\in \bF_q[x_1,\dots,x_n]$ be a polynomial of degree $d$ written as $f=F_d+F_{d-1}+\cdots+F_0$ with $F_i$ homogeneous of degree $i$. Suppose that the degree 
$d$ is relatively prime to $q$, and that the locus $F_d=0$ is a nonsingular hypersurface in $\bP^{n-1}$. Then
$$
\left| \sum_{x \in \bF_q^n} \psi(f(x)) \right| 
    \leq (d-1)^n q^{\frac{n}{2}} .
$$
\end{thm}

We are in a position to deduce an upper bound for $G(A,B)$.

\begin{prop}\label{prop:matrix-Gauss}
Let $\rho_A$ denote the dimension of the radical ${\rm rad}\,A$ of $(V_A,Q_A)$.
    One has
    $$
    G(A,B) \ll  q^{\frac{n^2+\rho_A}{2}}.
    $$
\end{prop}

\begin{proof}
By Proposition \ref{prop:Kuroda}, the quadratic space $V_A$ has an orthogonal decomposition
\begin{equation}\label{eq-pf:quad-decomp}
V_A \cong W_A \perp H \perp {\rm rad}\,A    
\end{equation}
where $W_A$ is regular of dimension $w_A$, $H$ is a hyperbolic space of dimension $2h_A$, and the radical ${\rm rad}\,A$ has dimension $\rho_A$. Hence 
$$
n^2=w_A+2h_A+\rho_A.
$$

It follows from the decomposition \eqref{eq-pf:quad-decomp} that, after a change of variables, there exist:
\begin{itemize}
    \item a nonsingular homogeneous quadratic polynomial 
$$ f_2(\bfx) \in \bF_q[\bfx] $$ 
in $w_A$ variables $\bfx=(x_1,\dots,x_{w_A})$, 
    \item a linear form $f_1(\bfx) \in \bF_q[\bfx]$, 
    \item the quadratic polynomial 
    $$ g_2(\bfy,\bfz) =  \sum_{i=1}^{h_A} y_iz_i $$
    where $\bfy=(y_1,\dots,y_{h_A})$ and $\bfz=(z_1,\dots,z_{h_A})$,
    \item a linear form $g_1(\bfy,\bfz) \in \bF_q[\bfy,\bfz]$,
    \item a linear form $r_1(\bfu)\in \bF_q[\bfu]$ in $\rho_A$ variables $\bfu=(u_1,\dots,u_{\rho_A})$,
\end{itemize}
such that
$$
G(A,B) = 
    \( \sum_{\bfx\in \bF_q^{w_A}} \psi\( f_2(\bfx) +  f_1(\bfx) \) \)
    \( \sum_{\bfy,\bfz\in \bF_q^{h_A}} \psi\( g_2(\bfy,\bfz) +  g_1(\bfy,\bfz) \) \)
    \( \sum_{\bfu\in \bF_q^{\rho_A}} \psi\( r_1(\bfu) \) \).
$$
By virtue of Theorem \ref{thm:Deligne}, the first factor is $O(q^{\frac{w_A}{2}})$. It is plain that the second factor is $O(q^{h_A})$, and that the third factor is $O(q^{\rho_A})$.
Thus
$$
G(A,B) \ll q^{\frac{w_A}{2}+h_A+\rho_A} =  q^{\frac{n^2+\rho_A}{2}}.
$$
The proposition is proved.
\end{proof}

\begin{rem}
Proposition \ref{prop:matrix-Gauss} reduces the estimation of $G(A,B)$ to that of the radical invariant $\rho_A$. 
By \cite[Theorem 2]{Kuroda04}, we have $|G(A,0)|=q^{\frac{n^2+\rho_A}{2}}$ . 
An interesting problem is to evaluate $G(A,B)$ for all $A,B\in\mn$. 
\end{rem}

\subsection{Quadratic cycle types and quadratic class types}\label{subsect:3-quad-types}

In this section, we introduce the notions of quadratic cycle type and quadratic class type, aiming to extract information about radical invariants most succinctly.

Let $\mn:=\mnq$ and $I:={\rm Irr}(\bF_q[T])$. Recall that the cycle type, or more generally the class type, of a matrix $A\in\mn$ encodes the size of its similarity class. The cycle type of $A$ can be expressed more refinedly as a four-part formal product
\begin{equation}\label{eq:quad-cyc-type}
\nu_A = T^\alpha \cdot
    \prod_\zeta \left( 
    \zeta(T)^{\beta_\zeta} 
    \zeta(-T)^{\gamma_\zeta} 
    \right) \cdot
    \prod_\pi \pi(T^2)^{\lambda_\pi} \cdot
    \prod_\eta \eta(T)^{\kappa_\eta} .    
\end{equation}
Here $\zeta$ varies in $I$ such that each $\zeta$ is not a polynomial in $T^2$, $\pi$ varies in $I$, and $\eta$ varies in $I$ such that any two (not necessarily distinct) irreducibles $\eta_1$ and $\eta_2$ satisfy $\eta_1(-T) \neq \eta_2(T)$.   
Correspondingly, the exponents  $\alpha,\beta_\zeta,\gamma_\zeta,\lambda_\pi,\kappa_\eta$ are partitions. We shall refer to \eqref{eq:quad-cyc-type} as the \emph{quadratic cycle type} of the matrix $A$. Forgetting the irreducible polynomials and retaining only their degrees, one arrives at the formal product
\begin{equation}\label{eq:quad-class-type}
\tau_A = 0^\alpha \cdot
    \prod_\zeta \left(  
    \deg(\zeta)^{\beta_\zeta} 
    \deg(\zeta)^{\gamma_\zeta} 
    \right)_{(+1)} \cdot
    \prod_\pi \left( \deg(\pi)^{\lambda_\pi} \right)_{(2)} \cdot
    \prod_\eta \left( \deg(\eta)^{\kappa_\eta} \right)_{(-1)},   
\end{equation}
called the \emph{quadratic class type} of the matrix $A$.

\subsection{Incidence functions}\label{subsect:incidence-functions}

We follow the argument of Iosevich and Rudnev \cite{IR07}.

Let $\mn:=\mnq$ and consider $E\subseteq\mn^r$.
Define the \emph{incidence function}
$$
\nu(T)=\frac{1}{(\# E)^2} 
    \#\{(\bfX,\bfY) \in E \times E: d(\bfX,\bfY)=T\}
    \quad (T\in\mn).
$$
For $T\in\mn$, the \emph{`matrix sphere' of `radius' $T$} is 
$$
\sigma_T=\{\bfX\in\mn^r:\|\bfX\|=T\} .
$$ 
On applying the Fourier inversion formula as in \cite[Section 2.2]{IR07}, we deduce that
\begin{equation}\label{eq:IR:incidence}
     \nu(T) = \frac{(\# \sigma_T)}{q^{rn^2}}
        + \frac{q^{2rn^2}}{(\# E)^2} \sum_{\bfM\neq 0} 
        \left| \widehat{E}(\bfM) \right|^2 \widehat{\sigma}_T(\bfM).
\end{equation}

\begin{lem}\label{lem:matrix-sphere}
If $T\in\mn$ and $\bfM\in \mn^r\setminus\{0\}$, then
    \begin{equation}
        \# \sigma_T
        = q^{n^2(r-1)} + 
        \frac{1}{q^{n^2}} 
        \sum_{S\in\mn\setminus\{0\}} 
            \psi\(-ST\) G(S,0)^r 
        \label{eq:IR:sphere-size}
    \end{equation} 
and
    \begin{equation}
        \widehat{\sigma}_T(\bfM)  
        = \frac{1}{q^{(r+1)n^2}} 
        \sum_{S\in\mn\setminus\{0\}} \psi(-ST) 
        \prod_{i=1}^r G(S,-M_i). 
        \label{eq:IR:FT-sphere-nonzero-phase}
    \end{equation} 
\end{lem}

\begin{proof}
We compute the size of a matrix sphere:
\begin{align*}
\# \sigma_T
    &= \frac{1}{q^{n^2}} \sum_{\bfX\in\mn^r} \sum_{S\in\mn} 
        \psi\(S\(\|\bfX\|-T\)\)
        \\ 
    &= q^{n^2(r-1)} + 
        \frac{1}{q^{n^2}} \sum_{\bfX=(X_1,\dots,X_r)\in\mn^r} 
            \sum_{S\in\mn\setminus\{0\}} 
            \psi\(S\(\|\bfX\|-T\)\)
        \\ 
    &= q^{n^2(r-1)} + 
        \frac{1}{q^{n^2}} \sum_{S\in\mn\setminus\{0\}} 
            \psi\(-ST\) G(S,0)^r. 
\end{align*}

We compute the Fourier transform a matrix sphere at $\bfM=(M_1,\dots,M_r)\in\mn^r$:
\begin{align*}
\widehat{\sigma}_T(\bfM)  
    &= \frac{1}{q^{rn^2}} \sum_{\bfX\in\mn^r} 
        \sigma_T(\bfX)\overline{\psi}(\bfM\cdot\bfX) 
        \\ 
    &= \frac{1}{q^{(r+1)n^2}} 
        \sum_{\bfX=(X_1,\dots,X_r)\in\mn^r} \sum_{S\in\mn} 
            \overline{\psi}(\bfM\cdot\bfX)
            \psi\(S\(\|\bfX\|-T\)\)
        \\ 
    &= \frac{1}{q^{(r+1)n^2}} 
        \sum_{S\in\mn} \psi(-ST) 
        \prod_{i=1}^r G(S,-M_i). 
\end{align*}
If $\bfM\neq {\bf 0}$, we further have
\begin{equation*}
\widehat{\sigma}_T(\bfM)  
    = \frac{1}{q^{(r+1)n^2}} 
        \sum_{S\in\mn\setminus\{0\}} \psi(-ST) 
        \prod_{i=1}^r G(S,-M_i). 
\end{equation*}
The lemma is proved.
\end{proof}

\section{Rank 2 and rank 3 matrices}\label{sect:proofs-ranks-23}


\subsection{Matrix spheres in \texorpdfstring{${\rm M}_2(\bF_q)$}{TEXT}}\label{subsect:4-rank2-matrices}

Write ${\rm M}_2 = {\rm M}_2(\bF_q)$.

\subsubsection*{Quadratic types}

Set $I:={\rm Irr}(\bF_q[T])$.
A matrix in ${\rm M}_2$ has precisely one of the following quadratic cycle types:
\begin{enumerate}
    \item[${\rm (i)}$] 
        $T^{(1,1)}$
    \item[${\rm (ii)}$] 
        $T^{(2)}$
    \item[${\rm (iii)}$] 
        $T^{(1)}\cdot (T-\lambda)^{(1)} 
        \quad (\lambda\in \bF_q^\times)$
    \item[${\rm (iv)}$] 
        $(T+\lambda)^{(1)} (T-\lambda)^{(1)} \quad (\lambda\in \bF_q^\times)$
    \item[${\rm (v)}$] 
        $\pi(T^2)^{(1)}
        \quad (\pi\in I, \deg(\pi)=1)$
    \item[${\rm (vi)}$] 
        $(T-\lambda)^{(1)} (T-\lambda')^{(1)} \quad (\lambda,\lambda'\in \bF_q^\times, \lambda'\neq -\lambda)$
    \item[${\rm (vii)}$] 
        $\eta(T)^{(1)}
        \quad (\eta\in I, \deg(\eta)=2,\eta(T)\neq\eta(-T))$.
\end{enumerate}
These quadratic cycle types give rise to to the following quadratic class types:
\begin{enumerate}
    \item[${\rm (i)}$] 
        $0^{(1,1)}$
    \item[${\rm (ii)}$] 
        $0^{(2)}$
    \item[${\rm (iii)}$] 
        $0^{(1)}\cdot \left( 1^{(1)} \right)_{(-1)}$
    \item[${\rm (iv)}$] 
        $\left( 1^{(1)} 1^{(1)} \right)_{(+1)}$
    \item[${\rm (v)}$] 
        $\left( 1^{(1)} \right)_{(2)}$
    \item[${\rm (vi)}$] 
        $\left( 1^{(1)} 1^{(1)} \right)_{(-1)}$
    \item[${\rm (vii)}$] 
        $\left( 2^{(1)} \right)_{(-1)}$.
\end{enumerate}

\begin{table}[!h]
\begin{tabular}{|c|l|l|l|l|l|}
\hline
\multicolumn{1}{|l|}{\begin{tabular}[c]{@{}l@{}}Quadratic \\ class type\end{tabular}} & \begin{tabular}[c]{@{}l@{}}Radical \\ invariant\end{tabular} & \begin{tabular}[c]{@{}l@{}}Centralizer \\ size\end{tabular} & \begin{tabular}[c]{@{}l@{}}Similarity class \\ size\end{tabular} & \begin{tabular}[c]{@{}l@{}}No.~quadratic \\ cycle types\end{tabular} & \begin{tabular}[c]{@{}l@{}}Incidence\\ contribution\end{tabular} \\ \hline
${\rm (i)}$                                                                           & $4$                                                          & $\sim q^4$                                                       & $\sim 1$                                                              & $1$                                                                  & $q^{4r}$                                                         \\ \hline
${\rm (ii)}$                                                                          & $2$                                                          & $\sim q^2$                                                  & $\sim q^2$                                                       & $\asymp 1$                                                                  & $O(q^{3r+2})$                                                    \\ \hline
${\rm (iii)}$                                                                         & $1$                                                          & $\sim q^2$                                                  & $\sim q^2$                                                       & $\asymp q$                                                             & $O(q^{\frac{5r}{2}+3})$                                          \\ \hline
${\rm (iv)}$                                                                          & $2$                                                          & $\sim q^2$                                                  & $\sim q^2$                                                       & $\asymp q$                                                             & $O(q^{3r+3})$                                                    \\ \hline
${\rm (v)}$                                                                           & $2$                                                          & $\sim q^2$                                                  & $\sim q^2$                                                       & $\asymp q$                                                             & $O(q^{3r+3})$                                                    \\ \hline
${\rm (vi)}$                                                                          & $0$                                                          & $\sim q^2$                                                  & $\sim q^2$                                                       & $\asymp q^2$                                                           & $O(q^{2r+4})$                                                    \\ \hline
${\rm (vii)}$                                                                         & $0$                                                          & $\sim q^2$                                                  & $\sim q^2$                                                       & $\asymp q^2$                                                         & $O(q^{2r+4})$                                                    \\ \hline
\end{tabular}
\caption{\label{rank-2-table} Quadratic class types in ${\rm M}_2(\bF_q)$ and their invariants.}
\end{table}

\subsubsection*{Quadratic Gauss sums and matrix spheres of $2\times 2$ matrices}

Table \ref{rank-2-table} computes several invariants of quadratic class types in ${\rm M}_2(\bF_q)$: the radical invariant, the centralizer size, the similarity class size, and the number of quadratic cycle types which have a particular quadratic class type.
In the last column of Table \ref{rank-2-table}, we estimate the contribution of each quadratic class type to the sums over $S$ in \eqref{eq:IR:sphere-size} and \eqref{eq:IR:FT-sphere-nonzero-phase}.

\begin{prop}\label{prop:matrix-sphere-rank2}
Suppose that $r\geq 4$. If $T\in {\rm M}_2$ and $\bfM\in {\rm M}_2^r\setminus\{{\bf 0}\}$, then
    \begin{equation}\label{eq:IR:sphere-size-rank2}
        \# \sigma_T \sim q^{4r-4} 
    \end{equation} 
and
    \begin{equation}
        \widehat{\sigma}_T(\bfM) = 0(q^{-r-1}).
        \label{eq:IR:FT-sphere-nonzero-phase-rank2}
    \end{equation} 
\end{prop}

\begin{proof}
    We first compute the size of the matrix sphere $\sigma_T$. By \eqref{eq:IR:sphere-size}, we have
    $$
    \# \sigma_T = q^{4(r-1)} + \frac{1}{q^{4}} 
            \sum_{S\in {\rm M}_2\setminus\{0\}} \psi\(-ST\) G(S,0)^r .
    $$
    By the column `Incidence contribution' of Table \ref{rank-2-table}, the sum over $S$ is $O(q^{3r+3})$.
    Since $r\geq 4$, we conclude that $\# \sigma_T \sim q^{4r-4}$. 

    We then compute Fourier transforms of the matrix sphere $\sigma_T$. For $\bfM\in {\rm M}_2^r\setminus\{{\bf 0}\}$, by \eqref{eq:IR:FT-sphere-nonzero-phase} we have
    $$
    \widehat{\sigma}_T(\bfM)  
        = \frac{1}{q^{4(r+1)}} 
        \sum_{S\in{\rm M}_2\setminus\{0\}} \psi(-ST) 
        \prod_{i=1}^r G(S,-M_i). 
    $$
    By the column `Incidence contribution' of Table \ref{rank-2-table}, the sum over $S$ is $O(q^{3r+3})$.
    Therefore $\widehat{\sigma}_T(\bfM) = 0(q^{-r-1})$. 
\end{proof}

\begin{proof}[Proof of Theorem \ref{thm:strong-Falconer-rank2}]
    It follows from \eqref{eq:IR:incidence} that 
    $$
    \nu(T) = \frac{(\# \sigma_T)}{q^{4r}}
        + \frac{q^{8r}}{(\# E)^2} \sum_{\bfM\neq {\bf 0}} 
        \left| \widehat{E}(\bfM) \right|^2 \widehat{\sigma}_T(\bfM).
    $$
    By the Plancherel formula we have
    $$
    \sum_{\bfM\neq {\bf 0}} 
        \left| \widehat{E}(\bfM) \right|^2
        \leq \sum_{\bfM\in {\rm M}_2^r} 
        \left| \widehat{E}(\bfM) \right|^2
        = \frac{(\# E)}{q^{4r}}.
    $$
    Therefore, by \eqref{eq:IR:FT-sphere-nonzero-phase-rank2},
    $$
    \nu(T) = \frac{(\# \sigma_T)}{q^{4r}} 
        + O\left( \frac{q^{3r-1}}{(\# E)} \right) .
    $$
    In view of \eqref{eq:IR:sphere-size-rank2}, $\nu(T)>0$ provided that
    $$
    (\# E) \gg_\epsilon q^{3r+3+\epsilon}.
    $$
\end{proof}

\subsection{Matrix spheres in \texorpdfstring{${\rm M}_3(\bF_q)$}{TEXT}}\label{subsect:4-rank3-matrices}

Write ${\rm M}_3 = {\rm M}_3(\bF_q)$.

\subsubsection*{Quadratic types}

Set $I:={\rm Irr}(\bF_q[T])$.
A matrix in ${\rm M}_3$ has precisely one of the following quadratic cycle types:
\begin{enumerate}
    \item[${\rm (i)}$] 
        $T^{(1,1,1)}$
    \item[${\rm (ii)}$] 
        $T^{(2,1)}$
    \item[${\rm (iii)}$] 
        $T^{(3)}$
    \item[${\rm (iv)}$] 
        $T^{(1,1)}\cdot (T-\lambda)^{(1)} 
        \quad (\lambda\in \bF_q^\times)$
    \item[${\rm (v)}$] 
        $T^{(2)}\cdot (T-\lambda)^{(1)} 
        \quad (\lambda\in \bF_q^\times)$
    \item[${\rm (vi)}$] 
        $T^{(1)} \cdot (T+\lambda)^{(1)} (T-\lambda)^{(1)} \quad (\lambda\in \bF_q^\times)$
    \item[${\rm (vii)}$] 
        $T^{(1)} \cdot \pi(T^2)^{(1)}
        \quad (\pi\in I, \deg(\pi)=1)$
    \item[${\rm (viii)}$] 
        $T^{(1)} \cdot (T-\lambda)^{(1)} (T-\lambda')^{(1)} \quad (\lambda,\lambda'\in \bF_q^\times, \lambda'\neq -\lambda)$
    \item[${\rm (ix)}$] 
        $T^{(1)} \cdot \eta(T)^{(1)}
        \quad (\eta\in I, \deg(\eta)=2,\eta(T)\neq\eta(-T))$
    \item[${\rm (x)}$] 
        $\eta(T)^{(1)}
        \quad (\eta\in I, \deg(\eta)=3)$
    \item[${\rm (xi)}$] 
        $(T-\lambda)^{(1)} \cdot \pi(T^2)^{(1)}
        \quad (\lambda\in \bF_q^\times, \pi\in I, \deg(\pi)=1)$
    \item[${\rm (xii)}$] 
        $(T-\lambda)^{(1)} \eta(T)^{(1)}
        \quad (\lambda\in \bF_q^\times, \eta\in I, \deg(\eta)=2,\eta(T)\neq\eta(-T))$
    \item[${\rm (xiii)}$] 
        $(T-\lambda_1)^{(1)} (T-\lambda_2)^{(1)} (T-\lambda_3)^{(1)} 
        \quad (\lambda_i \in \bF_q^\times, \lambda_i\neq -\lambda_j)$
    \item[${\rm (xiv)}$] 
        $(T-\lambda)^{(1)} (T+\lambda)^{(1)} \cdot (T-\lambda')^{(1)} 
        \quad (\lambda,\lambda'\in \bF_q^\times, \lambda'\neq \pm \lambda)$
    \item[${\rm (xv)}$] 
        $(T-\lambda)^{(2)} (T+\lambda)^{(1)} 
        \quad (\lambda \in \bF_q^\times)$
    \item[${\rm (xvi)}$] 
        $(T-\lambda)^{(1,1)} (T+\lambda)^{(1)} 
        \quad (\lambda \in \bF_q^\times)$.
\end{enumerate}
These quadratic cycle types give rise to to the following quadratic class types:
\begin{enumerate}
    \item[${\rm (i)}$] 
        $0^{(1,1,1)}$
    \item[${\rm (ii)}$] 
        $0^{(2,1)}$
    \item[${\rm (iii)}$] 
        $0^{(3)}$
    \item[${\rm (iv)}$] 
        $0^{(1,1)}\cdot \left( 1^{(1)} \right)_{(-1)}$
    \item[${\rm (v)}$] 
        $0^{(2)}\cdot \left( 1^{(1)} \right)_{(-1)}$
    \item[${\rm (vi)}$] 
        $0^{(1)}\cdot \left( 1^{(1)} 1^{(1)} \right)_{(+1)}$
    \item[${\rm (vii)}$] 
        $0^{(1)}\cdot \left( 1^{(1)} \right)_{(2)}$
    \item[${\rm (viii)}$] 
        $0^{(1)}\cdot \left( 1^{(1)} 1^{(1)} \right)_{(-1)}$
    \item[${\rm (ix)}$] 
        $0^{(1)}\cdot \left( 2^{(1)} \right)_{(-1)}$
    \item[${\rm (x)}$] 
        $\left( 3^{(1)} \right)_{(-1)}$
    \item[${\rm (xi)}$] 
        $\left( 1^{(1)} \right)_{(2)} \cdot \left(1^{(1)}\right)_{(-1)}$
    \item[${\rm (xii)}$] 
        $\left( 1^{(1)} 2^{(1)} \right)_{(-1)}$
    \item[${\rm (xiii)}$] 
        $\left( 1^{(1)} 1^{(1)} 1^{(1)} \right)_{(-1)}$
    \item[${\rm (xiv)}$] 
        $\left( 1^{(1)} 1^{(1)} \right)_{(+1)} 
        \cdot \left( 1^{(1)} \right)_{(-1)}$
    \item[${\rm (xv)}$] 
        $\left( 1^{(2)} 1^{(1)} \right)_{(+1)}$
    \item[${\rm (xvi)}$] 
        $\left( 1^{(1,1)} 1^{(1)} \right)_{(+1)}$.
\end{enumerate}

\subsubsection*{Quadratic Gauss sums and matrix spheres of $3\times 3$ matrices}

Table \ref{rank-3-table} computes invariants of quadratic class types in ${\rm M}_3(\bF_q)$. 
The last column of Table \ref{rank-3-table} estimates the contribution of each quadratic class type to the sums over $S$ in \eqref{eq:IR:sphere-size} and \eqref{eq:IR:FT-sphere-nonzero-phase}.

\begin{prop}\label{prop:matrix-sphere-rank3}
Suppose that $r\geq 3$. If $T\in {\rm M}_3$ and $\bfM\in {\rm M}_3^r\setminus\{{\bf 0}\}$, then
    \begin{equation}
        \# \sigma_T \sim q^{9r-9} \label{eq:IR:sphere-size-rank3}
    \end{equation} 
and
    \begin{equation}
        \widehat{\sigma}_T(\bfM) = 0(q^{-2r-5}).
        \label{eq:IR:FT-sphere-nonzero-phase-rank3}
    \end{equation} 
\end{prop}

\begin{proof}
    We first compute the size of the matrix sphere $\sigma_T$. By \eqref{eq:IR:sphere-size}, we have
    $$
    \# \sigma_T = q^{9(r-1)} + \frac{1}{q^{9}} 
            \sum_{S\in {\rm M}_3\setminus\{0\}} \psi\(-ST\) G(S,0)^r .
    $$
    By the column `Incidence contribution' of Table \ref{rank-3-table}, the sum over $S$ is $O(q^{7r+4})$.
    Since $r\geq 3$, we conlude that $\# \sigma_T \sim q^{9r-9}$. 

    We then compute Fourier transforms of the matrix sphere $\sigma_T$. For $\bfM\in {\rm M}_3^r\setminus\{{\bf 0}\}$, by \eqref{eq:IR:FT-sphere-nonzero-phase} we have
    $$
    \widehat{\sigma}_T(\bfM)  
        = \frac{1}{q^{9(r+1)}} 
        \sum_{S\in{\rm M}_3\setminus\{0\}} \psi(-ST) 
        \prod_{i=1}^r G(S,-M_i). 
    $$
    By the column `Incidence contribution' of Table \ref{rank-3-table}, the sum over $S$ is $O(q^{7r+4})$.
    Therefore $\widehat{\sigma}_T(\bfM) = 0(q^{-2r-5})$. 
\end{proof}

\begin{table}[H]
\begin{tabular}{|c|l|l|l|l|l|}
\hline
\multicolumn{1}{|l|}{\begin{tabular}[c]{@{}l@{}}Quadratic\\ class type\end{tabular}} & \begin{tabular}[c]{@{}l@{}}Radical\\ invariant\end{tabular} & \begin{tabular}[c]{@{}l@{}}Centralizer\\ size\end{tabular} & \begin{tabular}[c]{@{}l@{}}Similarity class\\ size\end{tabular} & \begin{tabular}[c]{@{}l@{}}No.~quadratic\\ cycle types\end{tabular} & \begin{tabular}[c]{@{}l@{}}Incidence\\ contribution\end{tabular} \\ \hline
${\rm (i)}$                                                                          & $9$                                                         & $\sim q^9$                                                      & $\sim 1$                                                             & $1$                                                                 & $q^{9r}$                                                         \\ \hline
${\rm (ii)}$                                                                         & $5$                                                         & $\sim q^5$                                                 & $\sim q^4$                                                      & $\asymp 1$                                                          & $O(q^{7r+4})$                                                    \\ \hline
${\rm (iii)}$                                                                        & $3$                                                         & $\sim q^3$                                                 & $\sim q^6$                                                      & $\asymp 1$                                                          & $O(q^{6r+6})$                                                    \\ \hline
${\rm (iv)}$                                                                         & $4$                                                         & $\sim q^5$                                                 & $\sim q^4$                                                      & $\asymp q$                                                          & $O(q^{\frac{13r}{2}+5})$                                         \\ \hline
${\rm (v)}$                                                                          & $2$                                                         & $\sim q^3$                                                 & $\sim q^6$                                                      & $\asymp q$                                                          & $O(q^{\frac{11r}{2}+7})$                                         \\ \hline
${\rm (vi)}$                                                                         & $3$                                                         & $\sim q^3$                                                 & $\sim q^6$                                                      & $\asymp q$                                                          & $O(q^{6r+7})$                                                    \\ \hline
${\rm (vii)}$                                                                        & $3$                                                         & $\sim q^3$                                                 & $\sim q^6$                                                      & $\asymp q$                                                          & $O(q^{6r+7})$                                                    \\ \hline
${\rm (viii)}$                                                                       & $1$                                                         & $\sim q^3$                                                 & $\sim q^6$                                                      & $\asymp q^2$                                                        & $O(q^{5r+8})$                                                    \\ \hline
${\rm (ix)}$                                                                         & $1$                                                         & $\sim q^3$                                                 & $\sim q^6$                                                      & $\asymp q^2$                                                        & $O(q^{5r+8})$                                                    \\ \hline
${\rm (x)}$                                                                          & $0$                                                         & $\sim q^3$                                                 & $\sim q^6$                                                      & $\asymp q^3$                                                        & $O(q^{\frac{9r}{2}+9})$                                          \\ \hline
${\rm (xi)}$                                                                         & $2$                                                         & $\sim q^3$                                                 & $\sim q^6$                                                      & $\asymp q^2$                                                        & $O(q^{\frac{11r}{2}+8})$                                         \\ \hline
${\rm (xii)}$                                                                        & $0$                                                         & $\sim q^3$                                                 & $\sim q^6$                                                      & $\asymp q^3$                                                        & $O(q^{\frac{9r}{2}+9})$                                          \\ \hline
${\rm (xiii)}$                                                                       & $0$                                                         & $\sim q^3$                                                 & $\sim q^6$                                                      & $\asymp q^3$                                                        & $O(q^{\frac{9r}{2}+9})$                                          \\ \hline
${\rm (xiv)}$                                                                        & $2$                                                         & $\sim q^3$                                                 & $\sim q^6$                                                      & $\asymp q^2$                                                        & $O(q^{\frac{11r}{2}+8})$                                         \\ \hline
${\rm (xv)}$                                                                         & $2$                                                         & $\sim q^3$                                                 & $\sim q^6$                                                      & $\asymp q$                                                          & $O(q^{\frac{11r}{2}+7})$                                         \\ \hline
${\rm (xvi)}$                                                                        & $4$                                                         & $\sim q^5$                                                 & $\sim q^4$                                                      & $\asymp q$                                                          & $O(q^{\frac{13r}{2}+5})$                                         \\ \hline
\end{tabular}
\caption{\label{rank-3-table} Quadratic class types in ${\rm M}_3(\bF_q)$ and their invariants.}
\end{table}

\begin{proof}[Proof of Theorem \ref{thm:strong-Falconer-rank3}]
    It follows from \eqref{eq:IR:incidence} that 
    $$
    \nu(T) = \frac{(\# \sigma_T)}{q^{9r}}
        + \frac{q^{18r}}{(\# E)^2} \sum_{\bfM\neq {\bf 0}} 
        \left| \widehat{E}(\bfM) \right|^2 \widehat{\sigma}_T(\bfM).
    $$
    By the Plancherel formula we have
    $$
    \sum_{\bfM\neq {\bf 0}} 
        \left| \widehat{E}(\bfM) \right|^2
        \leq \sum_{\bfM\in {\rm M}_3^r} 
        \left| \widehat{E}(\bfM) \right|^2
        = \frac{(\# E)}{q^{9r}}.
    $$
    Therefore, by \eqref{eq:IR:FT-sphere-nonzero-phase-rank3},
    $$
    \nu(T) = \frac{(\# \sigma_T)}{q^{9r}} 
        + O\left( \frac{q^{7r-5}}{(\# E)} \right) .
    $$
    In view of \eqref{eq:IR:sphere-size-rank3}, $\nu(T)>0$ provided that
    $$
    (\# E) \gg_\epsilon q^{7r+4+\epsilon}.
    $$
\end{proof}


\section{Matrices of rank at least 3}\label{sect:proofs-general-rank}

Throughout this section, we write ${\rm M}_n = \mnq$ and assume that $n\geq 3$.

Let $\tau$ be a given quadratic class type in $\mnq$.
For any matrix $A \in \mnq$ which has quadratic class type $\tau_A=\tau$, we want to estimate the following quantities:
\begin{itemize}
    \item the radical invariant $\rho=\rho_\tau=\rho_A$,
    \item the size $c=c_\tau=c_A$ of the centralizer in ${\rm GL}_n(\bF_q)$ of $A$, 
    \item the size $s=s_\tau=s_A$ of the similarity class of $A$, 
    \item and the number $y=y_\tau=y_A$ of quadratic cycle types which have quadratic class type $\tau$.
\end{itemize}

\subsection{Quadratic class type \texorpdfstring{$0^\alpha$}{TEXT}}\label{subsect:5-cycle-0}
 
Consider the quadratic class type $\tau=0^\alpha$ where
$$
\alpha = 1^{(e_1)} 2^{(e_2)} \dots k^{(e_k)} \quad (k \geq 1)
$$
is the partition of size $|\alpha|=\sum_{i=1}^k ie_i=n$ with $e_1$ numbers $1$, $e_2$ numbers $2$, \ldots, $e_k$ numbers $k$. 

When $k=1$, the invariants of the quadratic class type $\tau$ can be computed easily.

\begin{lem}\label{lem:quad-type-0-triv}
Let $\alpha = 1^{(n)}$ and $A\in\mn$ be a matrix with quadratic class type $\tau=0^\alpha$. Then $A$ is the zero matrix and we have 
\begin{enumerate}
    \item[{\rm (i)}] $\rho_\tau=n^2$;
    \item[{\rm (ii)}] $c_\tau\sim q^{n^2}$ and $s_\tau\sim 1$;
    \item[{\rm (iii)}] $y_\tau=1$;
    \item[{\rm (iv)}] if $r\in\bN$ and 
    $\cT=q^{\frac{r\rho_\tau}{2}}y_\tau s_\tau$, then 
    $\cT \sim q^{\frac{rn^2}{2}}$.
    \end{enumerate}
\end{lem}

\begin{proof}
    It is plain that $A$ is the zero matrix.
    Parts {\rm (i)}--{\rm (iii)} follow immediately from \eqref{eq:radical-invariant}, 
    \eqref{eq:cent-size-2} and the definitions of quadratic cycle type and quadratic class type. Part {\rm (iv)} can be readily derived from parts {\rm (i)}--{\rm (iii)}.
\end{proof}

\begin{lem}\label{lem:quad-type-0-nontriv}
Let
$$
\alpha = 1^{(e_1)} 2^{(e_2)} \dots k^{(e_k)} 
$$
with $k\geq 2$ and $\tau=0^\alpha$. We have:
\begin{enumerate}
    \item[{\rm (i)}] $\rho_\tau \leq n^2-2n+2$;
    \item[{\rm (ii)}] $c_\tau\sim q^{\rho_\tau}$, $s_\tau\sim q^{n^2-\rho_\tau}$;
    \item[{\rm (iii)}] $y_\tau=1$;
    \item[{\rm (iv)}] if $r\in\bN$ satisfies $r\geq 2$ and $\cT=q^{\frac{r\rho_\tau}{2}}y_\tau s_\tau$, then 
    $$
    \cT = O\(q^{\frac{rn^2}{2}-(n-1)(r-2)}\).
    $$
\end{enumerate}
\end{lem}

\begin{proof}
    For simplicity, abbreviate 
    $\rho=\rho_\tau, c=c_\tau,s=s_\tau,y=y_\tau$.
    Statement {\rm (iii)} is trivial.

    First note that $n=\sum_{i=1}^k ie_i$.
    It follows from \eqref{eq:radical-invariant} that
    \begin{equation}\label{eq-pf:rad-inv-form}
        \rho = \sum_{i=1}^k ie_i^2 + 2\sum_{1\leq i<j\leq k} ie_ie_j.
    \end{equation}
    Also note that, the identity \eqref{eq:cent-size-2} yields $c\sim q^d$ with
    $$
    d = \sum_{i=1}^{k} \(\sum_{j=i}^k e_j \)^2 = \rho .
    $$
    Hence {\rm (ii)} follows.

    We claim the following stronger inequality  
    \begin{equation}\label{eq-pf:rad-inv-ineq-main}
    \rho \leq n^2-(k-1)(2n-k);
    \end{equation}
    the bound {\rm (i)} is an immediate consequence of this claim. 
    To show the claim, we observe that \eqref{eq-pf:rad-inv-form} implies
    \begin{equation}\label{eq-pf:rad-inv-ineq}
        k\rho \leq n^2 +  (k-1)(n-ke_k)^2.
    \end{equation}
In fact, \eqref{eq-pf:rad-inv-ineq} is equivalent to
    \begin{equation}\label{eq-pf:rad-inv-ineq-1}
    \sum_{i=1}^k ike_i^2 + \sum_{1\leq  i<j\leq k} 2ik e_ie_j
        \leq  \(\sum_{i=1}^k ie_i\)^2  +  (k-1)\(\sum_{i=1}^{k-1} ie_i\)^2 . 
    \end{equation}
Each $e_i \,\, (1\leq i\leq k)$ viewed as a variable
and hence both sides of \eqref{eq-pf:rad-inv-ineq-1} viewed as quadratic forms, 
one readily verifies this inequality by expanding and checking that any quadratic form coefficient of the left hand side is at most the corresponding coefficient of the right hand side. 
    Since $e_k\geq 1$, we deduce from \eqref{eq-pf:rad-inv-ineq} that
    $$
    k\rho \leq n^2 + (k-1)(n-k)^2 ,
    $$
    from which the claim \eqref{eq-pf:rad-inv-ineq-main} follows. 
    The bound {\rm (i)} is proved.

    Finally, on combining {\rm (i)}--{\rm (iii)} we conclude {\rm (iv)}.  
\end{proof}

\subsection{Quadratic class type \texorpdfstring{$\prod_{u=1}^{l} \left(  
    d_u^{\beta_u} 
    d_u^{\gamma_u} 
    \right)_{(+1)}$}{TEXT}}\label{subsect:5-cycle+1}

Consider in $\mn$ the quadratic class type 
\begin{equation}\label{eq:quad-type-plus1-comp}
\tau = \prod_{u=1}^{l} \left(d_u^{\beta_u} d_u^{\gamma_u} \right)_{(+1)} .    
\end{equation}
In particular, $\sum_{u=1}^{l} d_u(|\beta_u|+ |\gamma_u|)=n$.

\begin{lem}\label{lem:quad-type-plus1}
In $\mn$ let
$$
\tau = \prod_{u=1}^{l} \left(d_u^{\beta_u} d_u^{\gamma_u} \right)_{(+1)}. $$
We have:
\begin{enumerate}
    \item[{\rm (i)}] $\rho_\tau \leq \frac{n^2}{4}$;
    \item[{\rm (ii)}] if $r\in\bN$ satisfies $r\geq 2$ and $\cT=q^{\frac{r\rho_\tau}{2}}y_\tau s_\tau$, then 
    $$
    \cT = O\(q^{\frac{rn^2}{8}+n^2}\).
    $$
\end{enumerate}
\end{lem}

\begin{proof}
    Let us write the partitions $\beta_u$ and $\gamma_u$ as
    \begin{align*}
        \beta_u &= 1^{\(e_{u,1}\)} 2^{\(e_{u,2}\)} 
                \ldots r_u^{\(e_{u,r_u}\)} , \\
        \gamma_u &= 1^{\(f_{u,1}\)} 2^{\(f_{u,2}\)} 
                \ldots s_u^{\(f_{u,s_u}\)} .
    \end{align*}
    Without loss of generality, we may assume that $r_u\leq s_u$.
    In view of \eqref{eq:radical-invariant}, we have 
    $\rho_\tau=\sum_{u=1}^{l} \rho_u$ where
    $$
    \rho_u := \sum_{i=1}^{r_u} i e_{u,i} f_{u,i} 
        + \sum_{\substack{1\leq i\leq r_u 
                \\ 1\leq j\leq s_u \\ i\neq j}} 
                \min(i,j) e_{u,i} f_{u,j}  .
    $$
    It follows that
    \begin{align*}
    \rho_u &\leq \sum_{i=1}^{r_u} \(i e_{u,i}\) f_{u,i} 
        + \sum_{\substack{1\leq i\leq r_u 
                \\ 1\leq j\leq s_u \\ i\neq j}} 
                \(i e_{u,i}\) f_{u,j}  \\
        &\leq \(\sum_{i=1}^{r_u} i e_{u,i}\) 
            \( \sum_{j=1}^{s_u} f_{u,j} \)    
        \leq \(\sum_{i=1}^{r_u} i e_{u,i}\) 
            \( \sum_{j=1}^{s_u} j f_{u,j} \) \\   
        &\leq \frac{1}{4} 
        \(\sum_{i=1}^{r_u} i e_{u,i} + \sum_{j=1}^{s_u} j f_{u,j} \)^2
        = \frac{1}{4} \( |\beta_u|+ |\gamma_u| \)^2.
    \end{align*}
    Therefore
    $$
    \rho_\tau=\sum_{u=1}^{l} \rho_u 
        \leq \frac{1}{4} \sum_{u=1}^{l} \( |\beta_u|+ |\gamma_u| \)^2
        \leq \frac{1}{4} \( \sum_{u=1}^{l} d_u \( |\beta_u|+ |\gamma_u| \) \)^2
        = \frac{n^2}{4}.
    $$
    Thus {\rm (i)} follows.
    
    Because the cardinality of $\mnq$ is $q^{n^2}$, it is evident that $s_\tau y_\tau \leq q^{n^2}$. Finally, we combine this bound with {\rm (i)} to conclude {\rm (ii)}.
\end{proof}

\subsection{Quadratic class type \texorpdfstring{$\prod_{u=1}^{l} \left( d_u^{\lambda_u} \right)_{(2)}$}{TEXT}}\label{subsect:5-cycle+2}

Consider in $\mn$ the quadratic class type 
\begin{equation}\label{eq:quad-type-2-comp}
\tau = \prod_{u=1}^{l} \left( d_u^{\lambda_u} \right)_{(2)}.    
\end{equation}
In particular, $2\sum_{u=1}^{l} d_u |\lambda_u| =n$.

\begin{lem}\label{lem:quad-type-2}
In $\mn$ let
$$
\tau = \prod_{u=1}^{l} \left( d_u^{\lambda_u} \right)_{(2)}.    
$$
We have:
\begin{enumerate}
    \item[{\rm (i)}] $\rho_\tau \leq \frac{n^2}{2}$;
    \item[{\rm (ii)}] if $r\in\bN$ satisfies $r\geq 2$ and $\cT=q^{\frac{r\rho_\tau}{2}}y_\tau s_\tau$, then 
    $$
    \cT = O\(q^{\frac{rn^2}{4}+n^2}\).
    $$
\end{enumerate}
\end{lem}

\begin{proof}
    For a partition $\lambda$, write $\rho(0^\lambda)=\rho_{0^\lambda}$ for the radical invariant of any matrix with the quadratic class type $0^\lambda$.
    By virtue of \eqref{eq:radical-invariant}, we have
    $$
    \rho_\tau =  2\sum_{u=1}^{l} d_u \rho(0^{\lambda_u}) .
    $$
    It follows that
    \begin{equation*}
    \rho_\tau \leq 2\sum_{u=1}^{l} d_u |\lambda_u|^2 
        \leq 2 \( \sum_{u=1}^{l} d_u |\lambda_u| \)^2 
        = \frac{n^2}{2}.
    \end{equation*}
    Thus {\rm (i)} follows.

    Because the cardinality of $\mnq$ is $q^{n^2}$, it is evident that $s_\tau y_\tau \leq q^{n^2}$. We then combine this bound with {\rm (i)} to conclude {\rm (ii)}.
\end{proof}

\subsection{Quadratic class type \texorpdfstring{$\prod_{u=1}^{l} \left( d_u^{\kappa_u} \right)_{(-1)}$}{TEXT}}\label{subsect:5-cycle-1}

Consider in $\mn$ the quadratic class type 
\begin{equation}\label{eq:quad-type-neg1-comp}
\tau = \prod_{u=1}^{l} \left( d_u^{\kappa_u} \right)_{(-1)}.    
\end{equation}
In particular, $\sum_{u=1}^{l} d_u |\kappa_u| =n$.

\begin{lem}\label{lem:quad-type-neg1}
In $\mn$ let
$$
\tau = \prod_{u=1}^{l} \left( d_u^{\kappa_u} \right)_{(-1)}.
$$
We have:
\begin{enumerate}
    \item[{\rm (i)}] $\rho_\tau = 0$;
    \item[{\rm (ii)}] if $r\in\bN$ 
    and $\cT=q^{\frac{r\rho_\tau}{2}}y_\tau s_\tau$, then 
    $$
    \cT \leq q^{n^2}.
    $$
\end{enumerate}
\end{lem}

\begin{proof}
    It follows from \eqref{eq:radical-invariant} that $\rho_\tau = 0$.

    Because the cardinality of $\mnq$ is $q^{n^2}$, it is evident that $s_\tau y_\tau \leq q^{n^2}$, whence {\rm (ii)}.
\end{proof}


\subsection{Matrix spheres in \texorpdfstring{${\rm M}_n(\bF_q), \, n\geq 3$}{TEXT}}\label{subsect:4-general-rank-matrices}


\begin{lem}\label{lem:gauss-sum-general-rank}
    Let $\tau$ be a nontrivial quadratic class type in $\mn$ with $n\geq 3$; in other words, $\tau \neq 0^{1^{(n)}}$.
    Suppose that $r\in\bN$ satisfies $r\geq 2$
    and let $\cT=q^{\frac{r\rho_\tau}{2}}y_\tau s_\tau$.
    Then 
    $$
    \cT = O\(q^{\frac{rn^2}{2}-(n-1)(r-2)}\).
    $$
\end{lem}

\begin{proof}
    Let us decompose $\tau$ as
    $$
    \tau = 0^\alpha 
    \cdot \prod_{u=1}^{l_1} 
        \left( d_u^{\beta_u} d_u^{\gamma_u}\right)_{(+1)} 
    \cdot \prod_{u=1}^{l_2} 
        \left( d_u^{\lambda_u} \right)_{(2)} 
    \cdot \prod_{u=1}^{l_3} 
        \left( d_u^{\kappa_u} \right)_{(-1)}.
    $$
    If $l_1=l_2=l_3=0$, then the lemma is none but part {(\rm iv)} of Lemma \ref{lem:quad-type-0-nontriv}.
    
    We may now assume that $l_1+l_2+l_3>0$.
    Put
    \begin{align*}
        n_0 &= |\alpha| \\
        n_1 &= \sum_{u=1}^{l_1} d_u \( |\beta_u| + |\gamma_u| \) \\
        n_2 &= 2 \sum_{u=1}^{l_2} d_u |\lambda_u| \\
        n_3 &= \sum_{u=1}^{l_3} d_u |\kappa_u|,
    \end{align*}
    so that $n=n_0+n_1+n_2+n_3$.
    On combining the decomposition of $\tau$ with the last parts of 
    Lemmas \ref{lem:quad-type-0-triv}, \ref{lem:quad-type-0-nontriv}, 
    \ref{lem:quad-type-plus1}, 
    \ref{lem:quad-type-2}, and
    \ref{lem:quad-type-neg1},
    we infer that 
    $\cT=\prod_{i=0}^{3}\cT_i$ where
    \begin{align*}
        \cT_0 &= O\( q^{\frac{rn_0^2}{2}} \) \\
        \cT_1 &= O\( q^{\frac{rn_1^2}{8}+n_1^2} \)  \\
        \cT_2 &= O\( q^{\frac{rn_2^2}{4}+n_2^2} \)  \\
        \cT_3 &= O\( q^{n_3^2} \)  .
    \end{align*}
    
    It remains to show that 
    \begin{equation}\label{eq-pf:form-reduced}
    \frac{rn_0^2}{2} + \frac{rn_1^2}{8}+n_1^2
    + \frac{rn_2^2}{4}+n_2^2 + n_3^2
    \leq \frac{rn^2}{2} - (n-1)(r-2).    
    \end{equation}
    Denote by $\cL$ (resp.~$\cR$) the left hand side (resp.~the right hand side) of \eqref{eq-pf:form-reduced}. We have
    \begin{equation}\label{eq-pf:form-reduced-1}
    \cL \leq \frac{rn_0^2}{2} + 
        \frac{r(n-n_0)^2}{4}+(n-n_0)^2.
    \end{equation}
    The inequality \eqref{eq-pf:form-reduced} then follows from 
    \eqref{eq-pf:form-reduced-1} and the following two inequalities, which can be verified easily:
    $$
    \frac{rn_0^2}{2} + 
        \frac{r(n-n_0)^2}{4}+(n-n_0)^2
        \leq \frac{r}{2}\( n^2 - 2n + \frac{3}{2}\) + 1
    $$
    and
    $$
     \frac{r}{2}\( n^2 - 2n + \frac{3}{2}\) + 1
     \leq \frac{rn^2}{2} - (n-1)(r-2).
    $$
    The proof is complete.
\end{proof}

\begin{prop}\label{prop:matrix-sphere-general-rank}
Suppose that $n\geq 3$ and $r\geq 3$. If $T\in {\rm M}_n$ and $\bfM\in {\rm M}_n^r\setminus\{{\bf 0}\}$, then
    \begin{equation}\label{eq:IR:sphere-size-general-rank}
        \# \sigma_T \sim q^{n^2(r-1)} 
    \end{equation} 
and
    \begin{equation}\label{eq:IR:FT-sphere-nonzero-phase-general-rank}
        \widehat{\sigma}_T(\bfM) = 0(q^{-(r-2)(n-1)-n^2}).
    \end{equation} 
\end{prop}

\begin{proof}
    We first compute the size of the matrix sphere $\sigma_T$. By \eqref{eq:IR:sphere-size}, we have
    \begin{align*}
    \# \sigma_T &= q^{n^2(r-1)} + \frac{1}{q^{n^2}} 
            \sum_{S\in {\rm M}_n\setminus\{0\}} \psi\(-ST\) G(S,0)^r \\
        &=: q^{n^2(r-1)} + \frac{\cS}{q^{n^2}} .
    \end{align*}
    
    By virtue of \eqref{prop:matrix-Gauss}, we have 
    $$
    G(S,0) \ll  q^{\frac{n^2+\rho_S}{2}}.
    $$
    To upper bound $\cS$, we partition the sum over $S$ into quadratic class types and apply Lemma \ref{lem:gauss-sum-general-rank}.
    We have 
    $$
    \cS \ll q^{rn^2-(n-1)(r-2)}.
    $$
    Thus \eqref{eq:IR:sphere-size-general-rank} follows.
    

    We then compute Fourier transforms of the matrix sphere $\sigma_T$. For $\bfM\in {\rm M}_n^r\setminus\{0\}$, by \eqref{eq:IR:FT-sphere-nonzero-phase} we have
    \begin{equation*}
    \widehat{\sigma}_T(\bfM)  
        = \frac{1}{q^{n^2(r+1)}} 
        \sum_{S\in{\rm M}_n\setminus\{0\}} \psi(-ST) 
        \prod_{i=1}^r G(S,-M_i) 
        =:  \frac{\cS'}{q^{n^2(r+1)}} .
    \end{equation*}
    To upper bound $\cS'$, we partition the sum over $S$ into quadratic class types and apply Lemma \ref{lem:gauss-sum-general-rank}.
    We have 
    $$
    \cS' \ll q^{rn^2-(n-1)(r-2)}.
    $$
    Thus \eqref{eq:IR:FT-sphere-nonzero-phase-general-rank} follows.
\end{proof}

We are in a position to prove the main result on the incidence function for a matrix algebra of rank at least 3.

\begin{proof}[Proof of Theorem \ref{thm:strong-Falconer-general-rank}]
    It follows from \eqref{eq:IR:incidence} that 
    $$
    \nu(T) = \frac{(\# \sigma_T)}{q^{rn^2}}
        + \frac{q^{2rn^2}}{(\# E)^2} \sum_{\bfM\neq {\bf 0}} 
        \left| \widehat{E}(\bfM) \right|^2 \widehat{\sigma}_T(\bfM).
    $$
    By the Plancherel formula we have
    $$
    \sum_{\bfM\neq {\bf 0}} 
        \left| \widehat{E}(\bfM) \right|^2
        \leq \sum_{\bfM\in {\rm M}_n^r} 
        \left| \widehat{E}(\bfM) \right|^2
        = \frac{(\# E)}{q^{rn^2}}.
    $$
    Therefore, by \eqref{eq:IR:FT-sphere-nonzero-phase-general-rank},
    $$
    \nu(T) = \frac{(\# \sigma_T)}{q^{rn^2}} 
        + O\left( \frac{q^{rn^2-(r-2)(n-1)-n^2}}{(\# E)} \right) .
    $$
    In view of \eqref{eq:IR:sphere-size-general-rank}, $\nu(T)>0$ provided that
    $$
    (\# E) \gg_\epsilon q^{rn^2-(r-2)(n-1)+\epsilon}.
    $$
\end{proof}




\bibliographystyle{amsplain}
\bibliography{erdos-falconer}
\end{document}